\documentclass[12pt]{article}

\usepackage{t1enc}
\usepackage[latin1]{inputenc}
\usepackage[english]{babel}
\usepackage{amsmath,amsthm}
\usepackage{amsfonts}
\usepackage{latexsym}
\usepackage{graphicx}
\usepackage{txfonts}
\usepackage[natural]{xcolor}
\usepackage{lineno}

\newtheorem{theorem}{Theorem}
\newtheorem{corollary}[theorem]{Corollary}

\newtheorem{lemma}[theorem]{Lemma}

\title
{On $r$-equitable chromatic threshold of Kronecker products of complete graphs\thanks{Supported by the
National Natural Science Foundation of China (No.\,11301410).}}

\author
{
Wei Wang$^{\rm a}$,
Zhidan Yan$^{\rm a}$,
Xin Zhang$^{\rm b}$\thanks{Corresponding author: xzhang@xidian.edu.cn.}
\\
{\footnotesize$^{\rm a}$College of Information Engineering, Tarim University, Alar 843300, P. R. China}\\
{\footnotesize$^{\rm b}$Department of Mathematics, Xidian University, Xi'an 710071, P. R. China}
}

\date{}

\begin{document}

 \maketitle
\begin{abstract}\baselineskip 0.60cm
A graph $G$ is $r$-equitably $k$-colorable if its vertex set can be partitioned into $k$ independent sets,
any two of which differ in size by at most $r$. The $r$-equitable chromatic threshold of a graph $G$, denoted by $\chi_{r=}^*(G)$, is the minimum $k$ such that $G$ is $r$-equitably $k'$-colorable for all $k'\ge k$. Let $G\times H$ denote the Kronecker product of graphs $G$ and $H$. In this paper, we completely determine the  exact value of $\chi_{r=}^*(K_m\times K_n)$ for general $m,n$ and $r$. As a consequence, we show that for $r\ge 2$, if $n\ge \frac{1}{r-1}(m+r)(m+2r-1)$ then $K_m\times K_n$ and its spanning supergraph $K_{m(n)}$ have the same $r$-equitable colorability, and in particular $\chi_{r=}^*(K_m\times K_n)=\chi_{r=}^*(K_{m(n)})$, where $K_{m(n)}$ is the complete $m$-partite graph with $n$ vertices in each part.\\[.2em]
\end{abstract}

\section{Introduction}\baselineskip 0.60cm
All graphs considered in this paper are finite, undirected and simple. Let $G$ be a graph with vertex set $V(G)$ and edge set $E(G)$. For a positive integer $k$, let $[k]=\{1,2,\ldots,k\}$. A (proper)  {\em $k$-coloring} of $G$ is a mapping $f:V(G) \to [k]$ such that
$f(x)\neq f(y)$ whenever $xy\in E(G)$. The {\em chromatic number} of $G$, denoted by $\chi(G)$, is the smallest integer $k$ such that $G$ admits a $k$-coloring. We call the set $f^{-1}(i)=\{x\in V(G)\colon f(x)=i\}$ a {\em color class} for each $i\in [k]$. Notice that each color class in a proper coloring is an independent set, i.e., a subset of $V(G)$ of pairwise non-adjacent vertices,
and hence a $k$-coloring is a partition of $V(G)$ into $k$ independent sets. For a fixed positive integer $r$,
an {\em $r$-equitable $k$-coloring} of $G$ is a $k$-coloring for which any two color classes differ in size by at most $r$.
A graph is {\em $r$-equitably $k$-colorable} if it has an $r$-equitable $k$-coloring.
The {\em $r$-equitable chromatic number} of $G$, denoted by $\chi_{r=}(G)$,
is the smallest integer $k$ such that $G$ is $r$-equitably $k$-colorable.
For a graph $G$, the {\em $r$-equitable chromatic threshold} of $G$, denoted by $\chi_{r=}^*(G)$,
is the smallest integer $k$ such that $G$ is $r$-equitably $k'$-colorable for all $k'\ge k$.
Although the concept of $r$-equitable colorability seems a natural generalization of usual equitable colorability (corresponding to $r$=1)  introduced by Meyer \cite{meyer1973} in 1973, it was first proposed in a recent paper by Hertz and Ries~\cite{hertz2011}, where the authors generalized the characterizations of usual equitable colorability of trees~\cite{chen1994} and forests~\cite{chang2009} to $r$-equitable colorability. Quite recently, Yen \cite{yen2013} proposed a necessary and sufficient condition for a complete multipartite graph $G$ to have an $r$-equitable $k$-coloring and also gave exact values of $\chi_{r=}(G)$ and $\chi_{r=}^*(G)$. In particular, they determined the value of $\chi_{r=}^*(K_{m(n)})$, where $K_{m(n)}$ denotes the complete $m$-partite graph with $n$ vertices in each part.
\begin{lemma}\cite{yen2013}\label{balancer}
For integers $n,r\ge 1$ and $m\ge 2$, we have $\chi_{r=}^*(K_{m(n)})=m\big\lceil\frac{n}{\theta+r}\big\rceil$, where $\theta$ is the minimum positive integer such that $\big\lfloor\frac{n}{\theta+1}\rfloor<\lceil\frac{n}{\theta+r}\rceil$.
\end{lemma}
The special case of Lemma \ref{balancer} for $r=1$ was obtained by Lin and Chang \cite{lin2010}.

For two graphs $G$ and $H$, the {\em Kronecker product} $G\times H$ of $G$ and $H$
is the graph with vertex set $\{(x,y)\colon\,x\in V(G),y\in V(H)\}$ and
edge set $\{(x,y)(x',y')\colon xx'\in E(G)$ and $yy'\in E(H)\}.$ In this paper, we pay attention to the $r$-equitable colorability of Kronecker product of two complete graphs.  We refer to \cite{chen2009,furmanczyk2006,lin2010,yan2013} for more studies on usual equitable colorability of Kronecker products of graphs.

In \cite{duffus1985}, Duffus et al.\,showed that if $m\le n$ then $\chi(K_m\times K_n)=m$. From this result, Chen \cite{chen2009} got that $\chi_=(K_m\times K_n)=m$ for $m\le n$. Indeed, let $V(K_m\times K_n)=\{(x_i,y_j)\colon\,i\in[m],j\in[n]\}$. Then we can partition $V(K_m\times K_n)$ into $m$  sets $\{(x_i,y_j)\colon\,j\in[n]\}$ with $i=1,2,\ldots,m$, all of which have equal size and are  clearly independent. Similarly, for any $r\ge 1$, $\chi_{r=}(K_m\times K_n)=m$ for  $m\le n$. However, it is much  more difficult to determine the exact value of $\chi_{r=}^*(K_m\times K_n)$, even for $r=1$.
\begin{lemma}\cite{lin2010}\label{upperbound1}
For positive integers $m\le n$, we have $\chi_=^*(K_m\times K_n)\le \big\lceil\frac{mn}{m+1}\big\rceil$.
\end{lemma}
In the same paper, Lin and Chang also determined $\chi_=^*(K_2\times K_n)$ and $\chi_=^*(K_3\times K_n)$. Note that the case when $m=1$ is trivial since $K_1\times K_n$ is the empty graph $I_n$ and hence $\chi_=^*(K_1\times K_n)=1$. Recently, those results have been improved to the following.

\begin{theorem}\cite{yan2014}\label{KmKnr1}
For integers $n\ge m\ge 2$,
\begin{equation*}
\chi_=^*(K_m\times K_n)=
\begin{cases} \big\lceil\frac{mn}{m+1}\big\rceil,&$if~$ n\equiv 2,\ldots,m-1~(\textup{mod}~m+1);\\
m\big\lceil\frac{n}{s^*}\big\rceil,&$if~$ n\equiv 0,1,m~(\textup{mod}~m+1),
\end{cases}
\end{equation*}
where $s^*$ is the minimum positive integer such that $s^*\nmid n$ and $m\big\lceil\frac{n}{s^*}\big\rceil\le  \big\lceil\frac{mn}{m+1}\big\rceil$.
\end{theorem}

From the definition of $s^*$, we see that $s^*\neq 1$ and hence $s^*\ge 2$. Let $\theta=s^*-1$. Then we can restate Theorem \ref{KmKnr1} as follows.
\begin{theorem}\label{resKmKnr1}
For integers $n\ge m\ge 2$,
\begin{equation*}
\chi_=^*(K_m\times K_n)=
\begin{cases} \big\lceil\frac{mn}{m+1}\big\rceil,&$if~$ n\equiv 2,\ldots,m-1~(\textup{mod}~m+1);\\
m\big\lceil\frac{n}{\theta+1}\big\rceil,&$if~$ n\equiv 0,1,m~(\textup{mod}~m+1),
\end{cases}
\end{equation*}
where $\theta$ is the minimum positive integer such that $\theta+1\nmid n$ and $m\big\lceil\frac{n}{\theta+1}\big\rceil\le  \big\lceil\frac{mn}{m+1}\big\rceil$.
\end{theorem}
\begin{corollary}\label{usu}
If $n\ge m$ and  $n\equiv 2,\ldots,m-1~(\textup{mod}~m+1)$ then $\chi_=^*(K_m\times K_n)<\chi_=^*(K_{m(n)})$.
\end{corollary}
\begin{proof}
Since $K_m\times K_n$ is a spanning subgraph of $K_{m(n)}$, $\chi_=^*(K_m\times K_n)\le \chi_=^*(K_{m(n)})$. Therefore, the lemma follows if we can show $\chi_=^*(K_m\times K_n)\neq \chi_=^*(K_{m(n)})$.
Let $n=(m+1)s+t$ with $s=\big\lfloor\frac{n}{m+1}\big\rfloor$ and $2\le t\le m-1$.  We have $\big\lceil\frac{mn}{m+1}\big\rceil=\big\lceil\frac{m(m+1)s+mt}{m+1}\big\rceil=\big\lceil\frac{m(m+1)s+(m+1)t-t}{m+1}\big\rceil
=ms+t+\big\lceil\frac{-t}{m+1}\big\rceil=ms+t$. By Theorem \ref{resKmKnr1}, $\chi_=^*(K_m\times K_n)=\big\lceil\frac{mn}{m+1}\big\rceil=ms+t$ and hence $m$ is not a factor of $\chi_=^*(K_m\times K_n)$. On the other hand, by Lemma \ref{balancer}, $m$ is a factor of $\chi_=^*(K_{m(n)})$. Therefore, $\chi_=^*(K_m\times K_n)\neq \chi_=^*(K_{m(n)})$ and hence  the proof is complete.
\end{proof}
The main purpose of this paper is to obtain the exact value of $\chi_{r=}^*(K_m\times K_n)$ for any $r\ge 1$, which we state as the following theorem.
\begin{theorem}\label{main}
For any integers $n\ge m\ge 2$ and $r\ge 1$,
\begin{equation*}
\chi_{r=}^*(K_m\times K_n)=
\begin{cases} n-r\big\lfloor\frac{n}{m+r}\big\rfloor,&$if~$ n\equiv 2,\ldots,m-1~(\textup{mod}~m+r)~$and~$\\ &\Big\lceil\frac{n}{\lfloor n/(m+r)\rfloor}\Big\rceil-\Big\lfloor\frac{n}{\lceil n/(m+r)\rceil}\Big\rfloor>r;\\
m\big\lceil\frac{n}{\theta+r}\big\rceil,&$otherwise,$
\end{cases}
\end{equation*}
where $\theta$ is the minimum positive integer such that $\big\lfloor\frac{n}{\theta+1}\big\rfloor<\big\lceil\frac{n}{\theta+r}\big\rceil$ and $m\big\lceil\frac{n}{\theta+r}\big\rceil\le \textup{min}\{n-r\big\lfloor\frac{n}{m+r}\big\rfloor,m\big\lceil\frac{n}{m+r}\big\rceil\}$.
\end{theorem}
Theorem \ref{main} agrees with Theorem \ref{resKmKnr1} when $r=1$. Firstly, $n-\big\lfloor\frac{n}{m+1}\big\rfloor=n+\big\lceil\frac{-n}{m+1}\big\rceil=\big\lceil\frac{(m+1)n-n}{m+1}\big\rceil= \big\lceil\frac{mn}{m+1}\big\rceil$. Secondly, we claim that $n\equiv 2,\ldots,m-1~(\textup{mod}~m+1)$ implies $\Big\lceil\frac{n}{\lfloor n/(m+1)\rfloor}\Big\rceil-\Big\lfloor\frac{n}{\lceil n/(m+1)\rceil}\Big\rfloor>1$. Let $n=(m+1)s+t$ with $s=\big\lfloor\frac{n}{m+1}\big\rfloor$ and  $2\le t\le m-1$. Then $(m+1)s<n<(m+1)(s+1)$ and hence
$$\bigg\lceil\frac{n}{\lfloor n/(m+1)\rfloor}\bigg\rceil-\bigg\lfloor\frac{n}{\lceil n/(m+1)\rceil}\bigg\rfloor=\Big\lceil\frac{n}{s}\Big\rceil-\Big\lfloor\frac{n}{s+1}\Big\rfloor\ge (m+2)-m\ge 2.$$
Finally, we need to check that two definitions of $\theta$ in Theorems \ref{resKmKnr1} and \ref{main} are equivalent. Clearly, $\big\lfloor\frac{n}{\theta+1}\big\rfloor<\big\lceil\frac{n}{\theta+1}\big\rceil$ if and only if $\theta+1\nmid n$. Since $m\big\lceil\frac{n}{m+1}\big\rceil$ is an integer and $m\big\lceil\frac{n}{m+1}\big\rceil\ge \frac{mn}{m+1}$, we have $m\big\lceil\frac{n}{m+1}\big\rceil\ge \big\lceil\frac{mn}{m+1}\big\rceil$. As we have already shown $n-\big\lfloor\frac{n}{m+1}\big\rfloor= \big\lceil\frac{mn}{m+1}\big\rceil$, we see that $\textup{min}\{n-\big\lfloor\frac{n}{m+1}\big\rfloor,m\big\lceil\frac{n}{m+1}\big\rceil\}=\big\lceil\frac{mn}{m+1}\big\rceil$. This shows the two definitions of $\theta$ are equivalent.


For fixed integers $m$ and $r\ge 2$,  Theorem \ref{main} can be simplified when $n$ is sufficiently large. Compared to Corollary \ref{usu},  the following theorem indicates that the behaviors of $\chi_{r=}^*(K_{m(n)})$ and $\chi_{r=}^*(K_m\times K_n)$ with $r\geq 2$ are quite different from the case when $r=1$.

\begin{theorem}\label{equ}
For any integers $n\ge m\ge 2$ and $r\ge 2$, if $n\ge\frac{1}{r-1}(m+r)(m+2r-1)$ then $\chi_{r=}^*(K_m\times K_n)=\chi_{r=}^*(K_{m(n)})$, and moreover, $K_m\times K_n$ and $K_{m(n)}$ have the same $r$-equitable colorability, that is,
 $K_m\times K_n$ is $r$-equitably $k$-colorable if and only if $K_{m(n)}$ is $r$-equitably $k$-colorable.
\end{theorem}

\section{Proof of Theorems \ref{main} and  \ref{equ}}
\begin{lemma}\cite{yen2013}\label{empty}
For any integer $r\ge 1$, $I_n$ has an $r$-equitable $k$-coloring with color classes of sizes between $m$ and $m+r$ if and only  $mk\le n\le(m+r)k$.
\end{lemma}
Lemma~\ref{empty} is slightly different from Lemma 7 in \cite{yen2013}. However, the original proof also applies in this statement.
\begin{lemma}\cite{yen2013}\label{yendis}
For integers $n,r\ge 1$ and $k\ge m\ge 2$, $K_{m(n)}$ is $r$-equitably $k$-colorable if and only if $\Big\lceil\frac{n}{\lfloor k/m\rfloor}\Big\rceil-\Big\lfloor\frac{n}{\lceil k/m\rceil}\Big\rfloor\le r.$
\end{lemma}
\begin{lemma} \label{inequ}
Let $m,n$ and $r$ be positive integers.\\
\textup{(1)} If $n\equiv 1,2,\ldots m-1~(\textup{mod}~m+r)$ then $n-r\big\lfloor\frac{n}{m+r}\big\rfloor<m\big\lceil\frac{n}{m+r}\big\rceil$.\\
\textup{(2)} If $n\equiv 0,m~(\textup{mod}~m+r)$ then $n-r\big\lfloor\frac{n}{m+r}\big\rfloor=m\big\lceil\frac{n}{m+r}\big\rceil$.\\
\textup{(3)} If $n\equiv m+1,\ldots,m+r-1~(\textup{mod}~m+r)$ then $n-r\big\lfloor\frac{n}{m+r}\big\rfloor>m\big\lceil\frac{n}{m+r}\big\rceil$.
\end{lemma}
\begin{proof}
Let $n=(m+r)s+t$ with $0\le t\le m+r-1$. Clearly, $n-r\big\lfloor\frac{n}{m+r}\big\rfloor=(m+r)s+t-rs=ms+t$ and
\begin{equation*}
m\Big\lceil\frac{n}{m+r}\Big\rceil=\begin{cases} ms&$if~$ t=0,\\
ms+m&$if~$ t=1,\ldots,m+r-1.
\end{cases}
\end{equation*}
Therefore, \\
(1) if $1\le t\le m-1$ then $n-r\big\lfloor\frac{n}{m+r}\big\rfloor=ms+t<ms+m=m\big\lceil\frac{n}{m+r}\big\rceil$,\\
(2) if $t=0$ or $t=m$ then $n-r\big\lfloor\frac{n}{m+r}\big\rfloor=ms+t=m\big\lceil\frac{n}{m+r}\big\rceil$,\\
(3) if $m+1\le t\le m+r-1$ then $n-r\big\lfloor\frac{n}{m+r}\big\rfloor=ms+t>ms+m=m\big\lceil\frac{n}{m+r}\big\rceil$.\\
This proves the lemma.
\end{proof}
The following lemma determines an upper bound for $\chi_{r=}^*(K_m\times K_n)$, a generalization of Lemma \ref{upperbound1}.
\begin{lemma}\label{upperboundr}
For positive integers $m\le n$ and $r$, we have $\chi_{r=}^*(K_m\times K_n)\le \textup{min}\{n-r\big\lfloor\frac{n}{m+r}\big\rfloor,m\big\lceil\frac{n}{m+r}\big\rceil\}$.
\end{lemma}
\begin{proof}
Let $\Gamma=\textup{min}\{n-r\big\lfloor\frac{n}{m+r}\big\rfloor,m\big\lceil\frac{n}{m+r}\big\rceil\}$ and let $k$ be any integer with $k\ge \Gamma$. We need to show that $K_m\times K_n$ is $r$-equitably $k$-colorable.  Noting  $\chi_{r=}^*(K_m\times K_n)\le \chi_{=}^*(K_m\times K_n)$ and $\big\lceil\frac{mn}{m+1}\big\rceil\le n$, Lemma~\ref{upperbound1} implies $\chi_{r=}^*(K_m\times K_n)\le n$. Therefore, we may assume further $k\le n$ and hence $\Gamma\le k\le n$.  Let $V(K_m\times K_n)=\{(x_i,y_j)\colon\, i\in[m],j\in[ n]\}$ and $n=(m+r)s+t$, where $s=\big\lfloor\frac{n}{m+r}\big\rfloor$.

{\em Case} 1. $n\equiv 0,1,\ldots,m~(\textup{mod}~m+r),$ i.e. $0\le t\le m$.

By Lemma \ref{inequ}, $n-r\big\lfloor\frac{n}{m+r}\big\rfloor\le m\big\lceil\frac{n}{m+r}\big\rceil$ and hence $\Gamma=n-r\big\lfloor\frac{n}{m+r}\big\rfloor=ms+t$. Let $V_j=\{(x_i,y_j)\colon\, i\in[ m]\}$ for  $1\le j\le k-ms$. By the definition of Kronecker products, each $V_j$ is an independent set. Let $n'=n-(k-ms)$. Since $ms+t=\Gamma\le k\le n$, we have $ms\le n'\le n-t=(m+r)s$. Let $U_i=\{(x_i,y_j)\colon\,k-ms+1\le j\le n\}$ for $i=1,2,\ldots,m$. Clearly each $U_i$ is an independent set of size $n'$. By Lemma~\ref{empty}, we can partition each $U_i$ with $i=1,2,\ldots,m$ into $s$ independent sets of sizes between $m$ and $m+r$. Combining $V_1,\ldots, V_{k-ms}$ with these $ms$ independent sets gives an $r$-equitable $k$-coloring of $K_m\times K_n$.

{\em Case} 2. $n\equiv m+1,\ldots,m+r-1~(\textup{mod}~m+r),$ i.e., $m+1\le t\le m+r-1$.

By Lemma \ref{inequ}, $\Gamma=m\big\lceil\frac{n}{m+r}\big\rceil=m(s+1)$ and hence $m(s+1)\le k\le n.$ Let $V_j=\{(x_i,y_j)\colon\,i\in [m]\}$ for $1\le j\le k-m(s+1)$. Clearly, each $V_j$ is an independent set of size $m$. Let $n'=n-(k-m(s+1)).$ Since $m(s+1)\le k\le n$, we have $m(s+1)\le n'\le n=(m+r)s+t\le (m+r)(s+1)$. Let $U_i=\{(x_i,y_j)\colon\,k-m(s+1)+1\le j\le n\}$ for $i=1,2,\ldots,m$. Clearly each $U_i$ is an independent set of size $n'$. By Lemma~\ref{empty}, we can partition each $U_i$ with $i=1,2,\ldots,m$ into $s+1$ independent sets of sizes between $m$ and $m+r$. Combining $V_1,\ldots,V_{k-m(s+1)}$ with these $m(s+1)$ independent sets gives an $r$-equitable $k$-coloring of $K_m\times K_n$.
\end{proof}
\begin{lemma}\label{minusi}
If $m,n,r$ and $\theta$ are positive integers with $m\ge 2$ and $\big\lfloor\frac{n}{\theta+1}\big\rfloor<\big\lceil\frac{n}{\theta+r}\big\rceil$, then $K_{m(n)}$ is not $r$-equitably $\big(m\big\lceil\frac{n}{\theta+r}\big\rceil-i\big)$-colorable for $1\le i<m$.
\end{lemma}
\begin{proof}
Let $q=\big\lceil\frac{n}{\theta+r}\big\rceil$. If $\theta+r\mid n$ then $\big\lceil\frac{n}{\theta+r}\big\rceil=\frac{n}{\theta+r}\le\frac{n}{\theta+1}$, yielding $\big\lceil\frac{n}{\theta+r}\big\rceil\le \lfloor\frac{n}{\theta+1}\big\rfloor$, a contradiction to the assumption of this lemma. Hence $\theta+r\nmid n$.  Now we have
$q=\big\lceil\frac{n}{\theta+r}\big\rceil>\frac{n}{\theta+1}\ge\frac{n}{\theta+r}>\big\lfloor\frac{n}{\theta+r}\big\rfloor=q-1.$ Consequently, $\frac{n}{q}<\theta+1$ and $\frac{n}{q-1}>\theta+r$. Note that we may assume $q-1\neq 0$ since  the lemma trivially follows  when $q=1$. Therefore, $\Big\lceil\frac{n}{\lfloor(mq-i)/m\rfloor}\Big\rceil- \Big\lfloor\frac{n}{\lceil(mq-i)/m\rceil}\Big\rfloor=\big\lceil\frac{n}{q-1}\big\rceil-\big\lfloor\frac{n}{q}\rfloor\ge(\theta+r+1)-\theta=r+1$ for $1\le i<m$. By Lemma \ref{yendis}, $K_{m(n)}$ is not $r$-equitably $\big(m\big\lceil\frac{n}{\theta+r}\big\rceil-i\big)$-colorable.
\end{proof}
\begin{lemma}\label{thetaprime}
For positive integers $m\ge 2$, $s,\theta,n$ and $r$, if $K_{m(n)}$ is not $r$-equitably $k$-colorable for some $k\ge m\big\lceil\frac{n}{\theta+r}\big\rceil$, then there is a positive integer $\theta'$ such that $\big\lfloor\frac{n}{\theta'+1}\big\rfloor<\big\lceil\frac{n}{\theta'+r}\big\rceil$, $\big\lceil\frac{n}{\theta'+r}\big\rceil=\big\lceil\frac{k}{m}\big\rceil$ and $\theta'<\theta$.
\end{lemma}
\begin{proof}
By Lemma \ref{yendis},  $\Big\lceil\frac{n}{\lfloor k/m\rfloor}\Big\rceil-\Big\lfloor\frac{n}{\lceil k/m\rceil}\Big\rfloor>r$. Hence,
$\frac{n}{\lfloor k/m\rfloor}>\theta'+r>\theta'+r-1>\cdots>\theta'+1>\frac{n}{\lceil k/m\rceil}$ for some nonnegative integer $\theta'$ and so
\begin{equation}\label{theta}
\Big\lceil\frac{k}{m}\Big\rceil>\frac{n}{\theta'+1}>\cdots>\frac{n}{\theta'+r}>\Big\lfloor\frac{k}{m}\Big\rfloor.
\end{equation}
 If $\theta'=0$ then the first inequality of (\ref{theta}) implies $k>mn$ and hence $K_{m(n)}$ is clearly  $r$-equitably $k$-colorable, a contradiction. Thus, $\theta'>0$. By (\ref{theta}), we see $\big\lceil\frac{k}{m}\big\rceil>\big\lfloor\frac{k}{m}\big\rfloor$ and hence $\big\lceil\frac{k}{m}\big\rceil=\big\lfloor\frac{k}{m}\big\rfloor+1$. Also from (\ref{theta}), we have
  $\big\lceil\frac{n}{\theta'+r}\big\rceil=\big\lceil\frac{k}{m}\big\rceil$ and $\big\lfloor\frac{n}{\theta'+1}\big\rfloor=\big\lfloor\frac{k}{m}\big\rfloor<\big\lceil\frac{n}{\theta'+r}\big\rceil$.
  Finally, $\frac{n}{\theta'+r}>\big\lfloor\frac{k}{m}\big\rfloor\ge \big\lfloor\frac{m}{m}\big\lceil\frac{n}{\theta+r}\big\rceil\big\rfloor=\big\lceil\frac{n}{\theta+r}\big\rceil\ge\frac{n}{\theta+r}$ implying $\theta'<\theta$.
\end{proof}
\begin{lemma}\label{mulkro}
For positive integers $m\ge 2$, $s,\theta,n$ and $r$, if $K_m\times K_n$ is $r$-equitably $k$-colorable for some $k<\textup{min}\{n-r\big\lfloor\frac{n}{m+r}\big\rfloor,m\big\lceil\frac{n}{m+r}\big\rceil\}$, then $K_{m(n)}$ is also $r$-equitably $k$-colorable.
\end{lemma}
\begin{proof}
Let $V(K_m\times K_n)=V(K_{m(n)})=\{(x_i,y_j)\colon\,i\in[m], j\in[n]\}$. Let $c$ be any $r$-equitable $k$-coloring of $K_m\times K_n$ with $k<\textup{min}\{n-r\big\lfloor\frac{n}{m+r}\big\rfloor,m\big\lceil\frac{n}{m+r}\big\rceil\}$. It suffices to show that each color class of $c$ is a subset of $\{(x_i,y_j)\colon\, j\in[n]\}$ for some $i\in [m]$. Let $\ell$ denote the number of color classes, each of which is a subset of $\{(x_i,y_j)\colon\,i\in[m]\}$ for some $j\in[n]$. Note that each independent set of $V(K_m\times K_n)$ is either a subset of $\{(x_i,y_j)\colon\,j\in[n]\}$ for some $i\in[m]$ or a subset of   $\{(x_i,y_j)\colon\, i\in[m]\}$ for some $j\in[n]$. Therefore, we only need to prove $\ell=0$.
Suppose to the contrary that $\ell>0$ and let $U_1,\ldots,U_\ell$ be such color classes defined above. Since any two color classes of $c$ differ in size by at most $r$ and some color class, say $U_1$, contains at most $m$ vertices, each color class is of size at most $m+r$. For each $i\in [m]$, let $k_i$ be the number of color classes contained in $W_i=\{(x_i,y_j)\colon\, j\in[n]\}\setminus \bigcup_{p=1}^\ell U_p$. Since $|W_i|\ge n-\ell$, we have
$k_i\ge\big\lceil\frac{n-\ell}{m+r}\big\rceil$. Therefore, $k=k_1+\cdots+k_m+\ell\ge m\big\lceil\frac{n-\ell}{m+r}\big\rceil+\ell$.

Define $a_q=m\big\lceil\frac{n-q}{m+r}\big\rceil+q$ for $q\ge 0$. Since $a_{q+m+r}=m\big\lceil\frac{n-q-m-r}{m+r}\big\rceil+q+m+r=m\big\lceil\frac{n-q}{m+r}\big\rceil+q+r=a_q+r>a_q$, the minimum of $\{a_q\colon\,q\ge 0\}$ exists and is achieved by $a_q$ for some $q\in\{0,1,\ldots,m+r-1\}$. Therefore, $k\ge a_\ell\ge \textup{min}\{a_0,a_1,\ldots,a_{m+r-1}\}$. Let $n=(m+r)s+t$ with $s=\big\lfloor\frac{n}{m+r}\big\rfloor$. Now, $a_q=  m\big\lceil\frac{n-q}{m+r}\big\rceil+q=ms+m\big\lceil\frac{t-q}{m+r}\big\rceil+q$. We shall show either of the following two cases yields a contradiction.

{\em Case} 1. $0\le t\le m-1.$

We claim in this case that $\textup{min}\{a_0,a_1,\ldots,a_{m+r-1}\}=ms+t$ and hence $k\ge ms+t$. Clearly, $a_t=ms+t$. If $0\le q\le t-1$ then $a_q=ms+m\big\lceil\frac{t-q}{m+r}\big\rceil+q\ge ms+m>ms+t$. If $t+1\le q\le m+r-1$ then $t-q\ge 0-(m+r-1)>-(m+r)$ and hence $a_q=ms+m\big\lceil\frac{t-q}{m+r}\big\rceil+q\ge ms+q>ms+t$. On the other hand, as shown in the proof of Lemma \ref{inequ}, we have $\textup{min}\{n-r\big\lfloor\frac{n}{m+r}\big\rfloor,m\big\lceil\frac{n}{m+r}\big\rceil\}=ms+t$. This is a contradiction to our assumption that $k<\textup{min}\{n-r\big\lfloor\frac{n}{m+r}\big\rfloor,m\big\lceil\frac{n}{m+r}\big\rceil\}$.

{\em Case} 2. $m\le t\le m+r-1.$

We claim in this case that $\textup{min}\{a_0,a_1,\ldots,a_{m+r-1}\}=m(s+1)$ and hence $k\ge m(s+1)$. Clearly,
$a_0=ms+m\big\lceil\frac{t}{m+r}\big\rceil=m(s+1).$ If $1\le q\le t-1$ then $a_q=ms+m\big\lceil\frac{t-q}{m+r}\big\rceil+q\ge ms+m+1>m(s+1)$. If $t\le q\le m+r-1$ then $a_q=ms+m\big\lceil\frac{t-q}{m+r}\big\rceil+q=ms+q\ge ms+t\ge m(s+1).$ Similarly, as shown in the proof of Lemma \ref{inequ}, we have $\textup{min}\{n-r\big\lfloor\frac{n}{m+r}\big\rfloor,m\big\lceil\frac{n}{m+r}\big\rceil\}=m(s+1)$, a contradiction.
\end{proof}
\noindent\textbf{Proof of Theorem \ref{main}.} Denote $\Gamma=\textup{min}\{n-r\big\lfloor\frac{n}{m+r}\big\rfloor,m\big\lceil\frac{n}{m+r}\big\rceil\}.$ We divide the proof into two cases.

{\em Case} 1. $n\equiv 2,\ldots,m-1~(\textup{mod}~m+r)$ and $\Big\lceil\frac{n}{\lfloor n/(m+r)\rfloor}\Big\rceil-\Big\lfloor\frac{n}{\lceil n/(m+r)\rceil}\Big\rfloor>r.$

By Lemmas \ref{upperboundr} and \ref{inequ}, $\chi_{r=}^*(K_m\times K_n)\le\Gamma=n-r\big\lfloor\frac{n}{m+r}\big\rfloor$.
Let $k=n-r\big\lfloor\frac{n}{m+r}\big\rfloor-1$. We need to show that  $K_m\times K_n$ is not $r$-equitably $k$-colorable. Noting $k<\Gamma$, it suffices to show that $K_{m(n)}$ is not $r$-equitably $k$-colorable by Lemma \ref{mulkro}.

By the first condition of this case, let $n=(m+r)s+t$ with $2\le t\le m-1$. We have $k=n-r\big\lfloor\frac{n}{m+r}\big\rfloor-1=ms+t-1$ and hence $ms<k<m(s+1)$. Consequently, $\big\lfloor\frac{k}{m}\big\rfloor=s$ and $\big\lceil\frac{k}{m}\big\rceil=s+1$. Since $\big\lfloor\frac{n}{m+r}\big\rfloor=s$ and $\big\lceil\frac{n}{m+r}\big\rceil=s+1$, we have
$\Big\lceil\frac{n}{\lfloor{k}/{m}\rfloor}\Big\rceil-\Big\lfloor\frac{n}{\lceil{k}/{m}\rceil}\Big\rfloor=
\big\lceil\frac{n}{s}\big\rceil-\big\lfloor\frac{n}{s+1}\big\rfloor=\Big\lceil\frac{n}{\lfloor n/(m+r)\rfloor}\Big\rceil-\Big\lfloor\frac{n}{\lceil n/(m+r)\rceil}\Big\rfloor>r$ from the last condition of this case. Therefore, by Lemma \ref{yendis}, $K_{m(n)}$ is not $r$-equitably $k$-colorable. This completes the proof of this case.

{\em Case} 2.  $n\equiv 0,1,m,m+1,\ldots,m+r-1~(\textup{mod}~m+r)$, or $n\equiv 2,\ldots,m-1~(\textup{mod}~m+r)$ and $\Big\lceil\frac{n}{\lfloor n/(m+r)\rfloor}\Big\rceil-\Big\lfloor\frac{n}{\lceil n/(m+r)\rceil}\Big\rfloor\le r.$

Since $\big\lfloor\frac{n}{\theta+1}\big\rfloor<\big\lceil\frac{n}{\theta+r}\big\rceil$, by Lemma \ref{minusi}, $K_{m(n)}$ is not $r$-equitably $\big(m\big\lceil\frac{n}{\theta+r}\big\rceil-1\big)$-colorable. Since $m\big\lceil\frac{n}{\theta+r}\big\rceil-1<
\textup{min}\{n-r\big\lfloor\frac{n}{m+r}\big\rfloor,m\big\lceil\frac{n}{m+r}\big\rceil\}$,
by Lemma \ref{mulkro}, $K_m\times K_n$ is not $r$-equitably $\big(m\big\lceil\frac{n}{\theta+r}\big\rceil-1\big)$-colorable. In the following, we prove that $K_m\times K_n$ is $r$-equitably $k$-colorable for all $k\ge m\big\lceil\frac{n}{\theta+r}\big\rceil$, which implies $\chi_{r=}^*(K_m\times K_n)=m\big\lceil\frac{n}{\theta+r}\big\rceil$.

Suppose to the contrary that $K_m\times K_n$ (and hence $K_{m(n)}$) is not $r$-equitably $k$-colorable for some $k\ge m\big\lceil\frac{n}{\theta+r}\big\rceil$. By Lemma \ref{upperboundr},
$k<\Gamma$. By Lemma \ref{thetaprime}, there is a positive integer $\theta'$ such that $\big\lfloor\frac{n}{\theta'+1}\big\rfloor<\big\lceil\frac{n}{\theta'+r}\big\rceil$, $\big\lceil\frac{n}{\theta'+r}\big\rceil=\big\lceil\frac{k}{m}\big\rceil$ and $\theta'<\theta$. By the minimality of $\theta$, $m\big\lceil\frac{n}{\theta'+r}\big\rceil>\Gamma$.

 Let  $n=(m+r)s+t$ with $0\le t\le m+r-1$. We show each of the following three subcases yields a contradiction.

{\em Subcase} 2.1. $n\equiv 0,1~(\textup{mod}~m+r)$, i.e., $t=0,1$.

By Lemma \ref{inequ}, $\Gamma=n-r\big\lfloor\frac{n}{m+r}\big\rfloor=ms+t$. Since $k<\Gamma$ we see $k<ms+t\le ms+1$, and hence $k\le ms.$  Therefore, $m\big\lceil\frac{n}{\theta'+r}\big\rceil=m\big\lceil\frac{k}{m}\big\rceil\le ms\le \Gamma$. This is a contradiction.

{\em Subcase} 2.2. $n\equiv m,\ldots,m+r-1~(\textup{mod}~m+r)$, i.e., $t=m,\ldots,m+r-1$.

By Lemma \ref{inequ}, $\Gamma=m\big\lceil\frac{n}{m+r}\big\rceil=m(s+1)$. Hence $k<m(s+1)$ and $m\big\lceil\frac{n}{\theta'+r}\big\rceil=m\big\lceil\frac{k}{m}\big\rceil\le m(s+1)\le \Gamma$. This is a contradiction.

{\em Subcase} 2.3.  $n\equiv 2,\ldots,m-1~(\textup{mod}~m+r)$ and $\Big\lceil\frac{n}{\lfloor n/(m+r)\rfloor}\Big\rceil-\Big\lfloor\frac{n}{\lceil n/(m+r)\rceil}\Big\rfloor\le r.$

By Lemma \ref{inequ}, $\Gamma=n-r\big\lfloor\frac{n}{m+r}\big\rfloor=ms+t$. If $k\le ms$ then $m\big\lceil\frac{n}{\theta'+r}\big\rceil=m\big\lceil\frac{k}{m}\big\rceil\le ms\le \Gamma$, a contradiction. Now assume  $k>ms$. Since $k<\Gamma=ms+t$, we have $ms<k<ms+t<m(s+1)$, yielding $\big\lfloor\frac{k}{m}\big\rfloor=s$ and $\big\lceil\frac{k}{m}\big\rceil=s+1$.  Consequently, by the second condition of this subcase, $\Big\lceil\frac{n}{\lfloor{k}/{m}\rfloor}\Big\rceil-\Big\lfloor\frac{n}{\lceil{k}/{m}\rceil}\Big\rfloor=
\big\lceil\frac{n}{s}\big\rceil-\big\lfloor\frac{n}{s+1}\big\rfloor=\Big\lceil\frac{n}{\lfloor n/(m+r)\rfloor}\Big\rceil-\Big\lfloor\frac{n}{\lceil n/(m+r)\rceil}\Big\rfloor\le r.$  Therefore, $K_{m(n)}$ is $r$-equitably $k$-colorable. This is a contradiction.\qed

\noindent \textbf{Proof of Theorem \ref{equ}.} Comparing Theorem \ref{main} with Lemma \ref{balancer}, it suffices to show, for the first part, that under the  assumption of this theorem, the following two statements hold:\\
  (\romannumeral1) $\Big\lceil\frac{n}{\lfloor n/(m+r)\rfloor}\Big\rceil-\Big\lfloor\frac{n}{\lceil n/(m+r)\rceil}\Big\rfloor\le r$,\\
   (\romannumeral2) if $\big\lfloor\frac{n}{\theta+1}\big\rfloor<\big\lceil\frac{n}{\theta+r}\big\rceil$ then $m\big\lceil\frac{n}{\theta+r}\big\rceil\le \textup{min}\{n-r\big\lfloor\frac{n}{m+r}\big\rfloor,m\big\lceil\frac{n}{m+r}\big\rceil\}$.

By the assumption that $n\ge\frac{1}{r-1}(m+r)(m+2r-1)$, we have $(r-1)\frac{n}{m+r}\ge m+2r-1$, yielding
  $$(r-1)\Big\lfloor\frac{n}{m+r}\Big\rfloor> (r-1)\frac{n}{m+r}-(r-1)\ge (m+2r-1)-(r-1)=m+r.$$  Multiplying the first and last term of the inequality by $\big\lceil\frac{n}{m+r}\big\rceil$ gives $$(r-1)\Big\lfloor\frac{n}{m+r}\Big\rfloor\Big\lceil\frac{n}{m+r}\Big\rceil>(m+r)\Big\lceil\frac{n}{m+r}\Big\rceil\ge n\ge \Big(\Big\lceil\frac{n}{m+r}\Big\rceil-\Big\lfloor\frac{n}{m+r}\Big\rfloor\Big)n.$$
Dividing by $\big\lfloor\frac{n}{m+r}\big\rfloor\big\lceil\frac{n}{m+r}\big\rceil$ leads to $\frac{n}{\lfloor n/(m+r)\rfloor}-\frac{n}{\lceil n/(m+r)\rceil}< r-1$. Hence, $\Big\lceil\frac{n}{\lfloor n/(m+r)\rfloor}\Big\rceil-\Big\lfloor\frac{n}{\lceil n/(m+r)\rfloor}\Big\rfloor<r+1$, which implies (\romannumeral1).

Now we assume further $\big\lfloor\frac{n}{\theta+1}\big\rfloor<\big\lceil\frac{n}{\theta+r}\big\rceil$ and show $m\big\lceil\frac{n}{\theta+r}\big\rceil\le \textup{min}\{n-r\big\lfloor\frac{n}{m+r}\big\rfloor,m\big\lceil\frac{n}{m+r}\big\rceil\}$. If $\frac{n}{\theta+1}-\frac{n}{\theta+r}\ge 1$ then $\big\lfloor\frac{n}{\theta+1}\big\rfloor\ge \big\lfloor\frac{n}{\theta+r}+1\big\rfloor\ge \big\lceil\frac{n}{\theta+r}\big\rceil$, a contradiction. Hence $\frac{n}{\theta+1}-\frac{n}{\theta+r}<1.$ Multiplying by $(\theta+1)(\theta+r)$ gives $(\theta+1)(\theta+r)>(r-1)n\ge (m+r)(m+2r-1)$, implying $\theta>m+r-1$. Hence  $m\big\lceil\frac{n}{\theta+r}\big\rceil\le m\big\lceil\frac{n}{m+r}\big\rceil$. It remains to show $m\big\lceil\frac{n}{\theta+r}\big\rceil\le n-r\big\lfloor\frac{n}{m+r}\big\rfloor$. Since $\theta>m+r-1$ and $n\ge\frac{1}{r-1}(m+r)(m+2r-1)$,  we have
\begin{eqnarray*}
m\Big\lceil\frac{n}{\theta+r}\Big\rceil+r\Big\lfloor\frac{n}{m+r}\Big\rfloor-n
&\le &m\Big\lceil\frac{n}{m+2r-1}\Big\rceil +\Big(r\frac{n}{m+r}-n\Big)\\
&\le& m\Big(1+\frac{n}{m+2r-1}\Big)-m\frac{n}{m+r}\\
&=&m\Big(1-\frac{(r-1)n}{(m+r)(m+2r-1)}\Big)\\
&\le&0,
\end{eqnarray*}
as desired.

Since $K_m\times K_n$ is a spanning subgraph of $K_{m(n)}$, $K_{m(n)}$ has an $r$-equitable $k$-coloring only if $K_m\times K_n$ has an $r$-equitable $k$-coloring. Suppose that $K_m\times K_n$ is $r$-equitably $k$-colorable for some integer $k$. If $k\ge \chi_{r=}^*(K_m\times K_n)$ then $k\ge \chi_{r=}^*(K_{m(n)})$, since $\chi_{r=}^*(K_m\times K_n)=\chi_{r=}^*(K_{m(n)})$, and hence $K_{m(n)}$ is $r$-equitably $k$-colorable.
If $k<\chi_{r=}^*(K_m\times K_n)$, then $k<\textup{min}\{n-r\big\lfloor\frac{n}{m+r}\big\rfloor,m\big\lceil\frac{n}{m+r}\big\rceil\}$
by Lemma \ref{upperboundr}. Therefore, Lemma \ref{mulkro} implies that $K_{m(n)}$ is $r$-equitably $k$-colorable. This completes the proof of Theorem \ref{equ}. \qed


\begin{thebibliography}{99}


\bibitem{chang2009} G. J. Chang, A note on equitable colorings of forests, European J. Combin. 30(2009) 809--812.

\bibitem{chen1994}  B.-L. Chen, K.-W. Lih, Equitable coloring of trees, J. Combin. Theory Ser. B 61 (1994) 83--87.

\bibitem{chen2009} B.-L. Chen, K.-W. Lih, J.-H. Yan, Equitable coloring of interval graphs and products of graphs, arXiv:0903.1396v1.

\bibitem{duffus1985} D. Duffus, B.Sands, R.E. Woodrow, On the chromatic number of the product of graphs, J. Graph Theory 9(4)(1985) 487--495

\bibitem{furmanczyk2006} H. Furma\'nczyk, Equitable colorings of graph products, Opuscula Math. 26 (2006), 31--44.


\bibitem{hertz2011} A. Hertz,  B. Ries,  On $r$-equitable colorings of trees and forests, Les Cahiers du GERAD (2011), G--2011--40.



\bibitem{lin2010} W.-H. Lin, G.J. Chang, Equitable colorings of Kronecker products of graphs, Discrete Appl. Math. 158 (2010) 1816--1826.

\bibitem{meyer1973} W. Meyer, Equitable coloring, Amer. Math. Monthly 80 (1973), 920--922.

\bibitem{yen2013} C.-H. Yen, On $r$-equitable coloring of complete multipartite graphs, Taiwanese J. Math. 13(2013),991--998.

\bibitem{yan2014} Z. Yan, W. Wang, Equitable chromatic threshold of direct products of
complete graphs, Ars Combin., to appear (accepted on 13 August 2013).

\bibitem{yan2013} Z. Yan, W. Wang, Equitable coloring of Kronecker products of complete multipartite graphs and complete graphs, Discrete Appl. Math. (2013), doi: 10.1016/j.dam.2013.08.042.


\end{thebibliography}
\end{document}